\documentclass[a4paper]{amsproc}

\newtheorem{theorem}{\rm\bf Theorem}[section]
\newtheorem{proposition}[theorem]{\rm\bf Proposition}
\newtheorem{lemma}[theorem]{\rm\bf Lemma}
\newtheorem{corollary}[theorem]{\rm\bf Corollary}

\newtheorem*{theorem 1}{\rm\bf Proposition 1}
\newtheorem*{theorem 2}{\rm\bf Proposition 2}

\theoremstyle{definition}

\theoremstyle{remark}
\newtheorem{remark}[theorem]{\rm\bf Remark}

\newtheorem{example}[theorem]{\rm\bf Example}

\def\scal#1#2{\langle #1, #2\rangle}
\def\R#1{\mathbb{R}^{#1}}

\def\TOP{\mathrm{t}}
\DeclareMathOperator{\trace}{\mathrm{tr}}
\DeclareMathOperator{\re}{\mathrm{Re}}

\begin{document}

%
\title{\uppercase{Non-isoparametric solutions of the eikonal equation}} 
\author{Vladimir Tkachev }

\address{Royal Institute of Technology, S-10044, Stockholm, Sweden}

 \email{tkatchev@kth.se}

\maketitle

%

\begin{abstract}
In this paper, we prove that a quartic solution of the eikonal equation $|\nabla_x f|^2=16x^{6}$ in $\R{n}$ is either isoparametric or congruent to a polynomial $f=(\sum_{i=1}^n x_i^2)^2-8(\sum_{i=1}^k x_i^2)(\sum_{i=k+1}^{n} x_i^2)$, $k=0,1,\ldots,[\frac{n}{2}]$.
\end{abstract}

\section{Introduction} 

Let $V= \R{n}$ be a Euclidean space equipped with the standard scalar product $\scal{\cdot}{\cdot}$. By $O(V)$ we denote the orthogonal group in $V$ and be $S(V)$ we denote the standard unit sphere in $V$ centered at the origin.
Recall that a hypersurface $M$ in $S(V)$ is called \textit{isoparametric} if it has constant principal curvatures \cite{Thorb}; see also a recent survey \cite{Cecil}. By a celebrated theorem of M\"{u}nzner \cite{Mun1}, any isoparametric hypersurface is  algebraic and its defining polynomial $f$ is homogeneous of degree $g=1,2,3,4$ or $6$, where $g$ is the number of distinct principal curvatures. Moreover, suitably normalized, $f$  satisfies the system of the so-called M\"{u}nzner-Cartan differential equations
\begin{equation}\label{Muntzer1}
|\nabla_x f|^2=g^2x^{2g-2}, \qquad \Delta_x f=\frac{m_2-m_1}{2}\,g^2x^{g-2},
\end{equation}
where $m_i$ are the multiplicities of the maximal and minimal principal curvature of $M$ (there holds $m_1=m_2$ when $g$ is odd). Here and in what follows, if no ambiguity is possible, we omit the norm notation by writing $x^{k}$ for $|x|^{k}$.

Isoparametric hypersurfaces of lower degrees $g=1,2,3$ were completely classified by \'{E}lie~ Cartan  in the late 1930-s. In the case of three distinct curvatures  one has $m_1=m_2$, so that any isoparametric cubic $f$ is a priori harmonic. In \cite{Cartan} Cartan established a remarkable result that there exist exactly four different isoparametric hypersurfaces with three distinct principal curvatures: their dimensions are equal to $3d$, where $d=1,2,4,8$, and the corresponding defining polynomials $f$ can be naturally expressed in terms of the multiplication in one of four real division algebras $\mathbb{F}_d$ of dimension $d$, where $\mathbb{F}_1=\mathbb{R}$ (reals), $\mathbb{F}_2=\mathbb{C}$ (complexes), $\mathbb{F}_4=\mathbb{H}$ (quaternions) and $\mathbb{F}_8=\mathbb{O}$ (octonions). The class isoparametric hypersurfaces with $g=4$ is very well understood by now, thanks to the works of
Ozeki and Takeuchi \cite{OT1}, \cite{OT2}, Ferus, Karcher, and M\"unzner \cite{FKM}, Cecil, Chi and Jensen \cite{CCJ} and several other authors.
However, in spite of much recent progress, isoparametric hypersurfaces with $g=4$ and $g=6$ distinct principal curvatures are not yet completely classified.

Regarding the M\"{u}nzner-Cartan differential equations as a system of two differential relations, a very natural question appears: How to characterize polynomial solutions of the first equation in (\ref{Muntzer1}) alone?
We shall call a homogeneous polynomial $f$ satisfying
\begin{equation}\label{Muntzer3}
|\nabla_x f|^2=g^2x^{2g-2}, \qquad \deg f=g, \quad x\in \R{n},
\end{equation}
an \textit{eikonal} polynomial.
This problem arises, for example, in geometrical optics and wave
propagation \cite{Rose}, \cite{Belousov}, the theory of  harmonic morphisms \cite{Baird}, entire solutions of the eikonal equation \cite{LiBQ}, transnormal hypersurfaces \cite{Robertson}, \cite{Bolton}.
Note also that any eikonal polynomial induces  a polynomial map $\mathrm{grad}_x f:S(\R{n})\to S(\R{n})$. Polynomial maps between Euclidean spheres, in particular those with harmonic coordinates (the so-called eigenmaps), has been the subject of considerable recent interest, with applications to the general problem of representing homotopy classes by polynomial maps   (see  \cite{Wood}, \cite{Baird} for the further discussion). Remarkably, for $n=4$ any  nonconstant quadratic eigenmap $F\colon S(\R{4})\to S(\R{4})$  up to isometries of the domain and the range, is the gradient of a Cartan cubic isoparametric polynomial \cite{Eigen}.

The characterization problem of eikonal polynomials, except for the trivial cases $g=1,2$ remains essentially open. Only recently, for $g=3$  the following complete description of  eikonal cubic polynomials was obtained in \cite{TkCartan}. We have proved that any eikonal polynomial for $g=3$ is congruent to either one of the four isoparametric Cartan cubic polynomials in dimensions $n=5,8,14$ and $26$, or to the polynomial
\begin{equation}\label{cubic}
h=x_n^3-3x_n(x_1^2+\ldots+x_{n-1}^2).
\end{equation}
Recall that two eikonal polynomials $f_1$ and $f_2$ are called {congruent} if there is an orthogonal transformation $U\in O(V)$ such that $f_2(x)=\pm f_1(Ux)$. Observe that $\Delta h=3(2-n)x_n$, thus $h$ is non-isoparametric for $n\ne 2$. Our proof is heavily based on one result of Yiu \cite{Yiu} on quadratic maps between Euclidean spheres and the theory of composition formulas \cite{Shapiro}.

In this paper, we  obtain a similar result for the eikonal quartics, that is the solutions of (\ref{Muntzer3}) with $g=4$. Before formulating this result we make some preliminary definitions and observations. The above family (\ref{cubic}) has a natural generalization for any degree $g$. Namely, let us  associate with an arbitrary  subspace $H$ of $V$ and an integer $g\ge 1$ the following function:
\begin{equation}\label{virtue}
h_{g,H}= \re (|\xi|+ |\eta|\sqrt{-1})^g=\sum_{k=0}^{[g/2]}(-1)^k\binom{g}{k}\xi^{g-2k} \eta^{2k},
\end{equation}
where $\xi$ and $\eta$ are the orthogonal projections of $x$ onto $H$ and $H^\bot=V\ominus H$ respectively. Then the cubic $h$ in (\ref{cubic}) is exactly $h_{3,H}$ with $H=\R{}e_n$. In general, one can easily verify that the following is true.

\begin{proposition}\label{pr:eik}
Suppose that either $g$ is even or $g$ is odd but $\dim H_1=1$. Then $h_{g,\mathcal{H}}$ is eikonal polynomial of degree $g$.
\end{proposition}

We  call a polynomial congruent to  $(\ref{virtue})$ a \textit{primitive} eikonal polynomial.

A typical \textit{quartic} primitive polynomial in the above notation reads as follows:
\begin{equation*}\label{virtue4}
h_{4,H}= \xi^4-6\xi^2\eta^2+\eta^4.
\end{equation*}

\begin{proposition}\label{pr:new}
Two primitive quartics $h_{4,H_1}$ and $h_{4,H_2}$ in $\R{n}$ are congruent if and only if either $\dim H_1=\dim H_2$ or $\dim H_1=n-\dim H_2$.
\end{proposition}

\begin{proof}
It suffices to prove the `only if' part. Suppose that $h_{4,H_1}$ and $h_{4,H_2}$ are congruent. Then there exists an orthogonal transformation $U\in O(V)$ such that $h_{4,H_2}(x)=\epsilon h_{4,H_1}(Ux)$, $\epsilon^2=1$. Since the Laplacian is an invariant operator, the quadratic forms $\Delta_x h_{4,H_1}$ and $\epsilon\Delta_x h_{4,H_2}$ have the same spectrum. We have
$$
\Delta_x h_{4,H_i}=(8+16p_i-12n)\xi_i^2+(4n-16p_i+8)\eta_i^2\equiv \lambda_i\xi_i^2+\mu_i\eta_i^2,
$$
where $\xi_i$ and $\eta_i$ are the projections on $H_i$ and $H_i^\bot$ respectively, and $p_i=\dim H_i$, $i=1,2$. Note that the quadratic forms are diagonal. If $\epsilon=1$ then a simple analysis implies that either $p_1=p_2$ or $p_1=n-p_2$. If $\epsilon=-1$ we have $\lambda_1+\mu_1=-\lambda_2-\mu_2$. On the other hand, $\lambda_i+\mu_i=16-8n$, hence $n=2$ and in this case the conclusion of the proposition is trivial.
\end{proof}

It follows from Proposition~\ref{pr:new} that there is exactly $[\frac{n}{2}]$ distinct congruence classes of primitive quartics in $\R{n}$. Any such a quartic is congruent to one of the following:
$$
f=(\sum_{i=1}^n x_i^2)^2-8(\sum_{i=1}^k x_i^2)(\sum_{i=k+1}^{n} x_i^2), \qquad 0\le k\le [\frac{n}{2}].
$$

The main result of the present paper is the following theorem.

\begin{theorem}
\label{th:polar}
Any quartic eikonal polynomial is either isoparametric or primitive.
\end{theorem}

The proof of Theorem~\ref{th:polar} will be given in the remaining sections of the paper. We conclude this Introduction by observing that, in view of the remarks made above, the following conjecture seems to be plausible.

\medskip
\textbf{Conjecture. }
An  eikonal polynomial of an arbitrary degree $g\ge 2$ is either isoparametric or primitive. In particular, if $g\ne2,3,4,6$ then $f$ is primitive.
\medskip

\section{Preliminaries}\label{sec:pre}

Note that any quartic homogeneous polynomial $f$ can be written in some orthogonal coordinates in the following \textit{normal} form:
\begin{equation}\label{reduced}
f=x_n^4+2\phi(\bar x)x_n^2+8\psi(\bar x)x_n+\theta(\bar x), \quad \bar x=(x_1,\ldots, x_{n-1}),
\end{equation}
where $\phi$, $\psi$ and $\theta$ are homogeneous polynomials of degrees 2,3 and 4 respectively.
To see this, it suffices to choose an orthonormal basis $e_1,\ldots,e_n$ of $V$ with  $e_n$ being a maximum point of the restriction of $f|_{S(V)}$.
Note that for an \textit{arbitrary} $f$, its normal form is by no means unique and the polynomials $\phi$, $\psi$ and $\theta$ carry no specific intrinsic information on $f$. As we shall see below, the situation with eikonal polynomials is remarkable and many basic algebraic properties of $f$ can be derived directly from the first component $\phi$.

Now we suppose that $f$ written in the normal form (\ref{reduced}) is eikonal, that is $f$ satisfies (\ref{Muntzer3}) for $g=4$. Then identifying  the coefficients of $x^j_n$ in (\ref{Muntzer3}) for $0\le j\le 4$ yields the following system:
\begin{equation}\label{eq1}
8\phi+|\nabla_{\bar x} \phi|^2=12\bar x^2,
\end{equation}
\begin{equation}\label{eq2}
\scal{\nabla_{\bar x} \phi}{\nabla_{\bar x} \psi}=-2\psi,
\end{equation}
\begin{equation}\label{eq3}
4\phi^2+\scal{\nabla_{\bar x} \phi}{\nabla_{\bar x} \theta}+16|\nabla_{\bar x} \psi|^2=12\bar x^4,
\end{equation}
\begin{equation}\label{eq4}
\scal{\nabla_{\bar x} \psi}{\nabla_{\bar x} \theta}=-4\phi\psi,
\end{equation}
\begin{equation}\label{eq5}
64\psi^2+|\nabla_{\bar x} \theta|^2=16\bar x^6.
\end{equation}
We can assume without loss of generality that the quadratic form  $\phi$ is diagonal, say $\phi=\sum_{i=1}^{n-1}\phi_ix_i^2$, so that (\ref{eq1}) implies  $\phi_i=1$ or $\phi_i=-3$. Let us denote by $L$ and $M$ the corresponding eigenspaces of $\bar V=V\ominus \R{}e_n$, and set $\dim L=p$ and $\dim M=q$. By denoting $\xi$ and $\eta$ the projections of $\bar x$ onto $L$ and $M$ respectively, we get
\begin{equation*}\label{contr}
\phi=\xi^2-3\eta^2, \qquad \bar x=(\xi,\eta).
\end{equation*}

Thus, with any eikonal quartic is associated an ordered pair of nonnegative integers $(p,q)$.  We shall indicate  this  by writing $f\in E_{p,q}$. Note that the pair $(p,q)$  is not uniquely determined as the following example shows.

\begin{example}\label{ex1}
Let us consider the one-dimensional subspace $H=\R{}e_1$ spanned on the vector $e_1$. Then the primitive eikonal polynomial
$$
f\equiv h_{4,H}=x_1^4-6x_1^2(x_2^2+\ldots+x_n^2)+(x_2^2+\ldots+x_n^2)^2
$$
is written in the normal form with $\phi=-3(x_2^2+\ldots+x_n^2)$, thus $f\in E_{0,n-1}$. On the other hand, one can rewrite $f$ as follows:
$$
f=x_{2}^4+2(x_3^2+\ldots x_{n}^2-3x_1^2)x_{2}^2+\theta,
$$
where $\theta$ is a quartic form in $(x_1,x_3,\ldots,x_n)$. This shows that $f\in E_{n-2,1}$.
\end{example}

Thus obtained  orthogonal decomposition $\bar V=L\oplus M$ induces the corresponding decompositions in the tensor products. We  write  $h\in \xi^i\otimes \eta^j$ if the polynomial $h$ is homogeneous in $\xi$ and $\eta$ of degrees $i$ and $j$ respectively. In particular,  by decomposing the cubic form $\psi$ into the sum of its homogeneous parts, $\sum_{i=0}^3\psi_{i}$,
where $\psi_{i}\in \xi^{i}\otimes \eta^{3-i}$, we find that
$$
\scal{\nabla_{\bar x} \phi}{\nabla_{\bar x} \psi}=2\sum_{i=0}^3 \scal{\xi}{\nabla_\xi \psi_i}-6\sum_{i=0}^3 \scal{\eta}{\nabla_\eta \psi_i}=
2\sum_{i=0}^3 (4i-9)\psi_i,
$$
hence by (\ref{eq2}), $\sum_{i=0}^3 (4i-8)\psi_i=0$. Taking into account that non-zero $\psi_i$ are linearly independent, we conclude that $\psi_0=\psi_1=\psi_3\equiv 0$. Thus, $\psi$ is completely determined by the component $\psi_2$ which is a quadratic form in $\xi$ and a linear form in $\eta$. In matrix notation this reads as follows:
\begin{equation}\label{psi}
\psi=\xi^\TOP A_\eta \xi,\qquad A_\eta:=\sum_{i=1}^q \eta_i A_i,
\end{equation}
where $A_i\in \R{p\times p}$ are symmetric matrices and $\xi^\TOP$ denotes the transpose to $\xi$.

Now we proceed with (\ref{eq3}). We have
$$
|\nabla_{\bar x}\psi|^2=|\nabla_\xi\psi|^2+|\nabla_\eta\psi|^2=4\xi^\TOP A_\eta^2\xi+\sum_{i=1}^q (\xi^\TOP A_i\xi)^2,
$$
and $\scal{\nabla_{\bar x} \phi}{\nabla_{\bar x} \theta}=8\sum_{i=0}^4 (i-3)\theta_i$, where we decomposed $\theta=\sum_{i=0}^4 \theta_i$ with $\theta_i\in \xi^i\otimes \eta^{4-i}$. Thus, (\ref{eq3}) takes the form
\begin{equation*}\label{eq30}
16(\xi^2-3\eta^2)^2+8\sum_{i=0}^4 (i-3)\theta_i+64\xi^\TOP A_\eta^2\xi+16\sum_{i=1}^q (\xi^\TOP A_i\xi)^2=12(\xi^2+\eta^2)^2.
\end{equation*}
The latter identity yields
$$
\theta_4-\theta_2-2\theta_1-3\theta_0=\xi^4-2\sum_{i=1}^q (\xi^\TOP A_i\xi)^2+(6\xi^2\eta^2-8\xi^\TOP A_\eta^2\xi)-3\eta^4,
$$
hence by identifying the homogeneous parts, we find that $\theta_1\equiv0$ and
\begin{equation}\label{theta}
\begin{split}
\theta_4&=\xi^4-2\sum_{i=1}^q (\xi^\TOP A_i\xi)^2,\\
\theta_2&=8\xi^\TOP A_\eta^2\xi-6\xi^2\eta^2,\\
\theta_0&=\eta^4.
\end{split}
\end{equation}

It follows from (\ref{theta}) and (\ref{psi}) that any solution of (\ref{Muntzer3}) $f$ is completely determined by the matrix pencil $A_\eta$ and the cubic form $\theta_3$. Our next step is to show that the matrix pencil satisfies a generalized Clifford property. To this end, we rewrite (\ref{eq4}) in the new notation:
\begin{equation}\label{eq40}
\scal{\nabla_\xi \psi}{\nabla_\xi\theta}+\scal{\nabla_\eta \psi}{\nabla_\eta\theta}=-4(\xi^2-3\eta^2)\xi^\TOP A_\eta\xi.
\end{equation}
We have $\nabla_\xi\psi=2A_\eta \xi$ and $\nabla_\eta\psi=\tau$, where
$$
\tau=(\xi^\TOP A_1\xi, \ldots, \xi^\TOP A_q\xi),
$$
hence we find from (\ref{theta}) the following gradients:
\begin{equation}\label{grad}
\begin{split}
\nabla_\xi\theta_2&=16A_\eta^2 \xi-12\eta^2 \cdot \xi,\\
\partial_{\eta_i}\theta_2&=16\xi^\TOP A_iA_\eta\xi -12\xi^2 \cdot \eta_i\\
\nabla_\xi\theta_4&=4\xi^2 \cdot \xi-8\sum_{i=1}^q (\xi^\TOP A_i\xi)A_i\xi\equiv
4\xi^2 \cdot \xi-8A_\tau \xi,\\
\nabla_\eta\theta_4&=4\eta^2 \cdot \eta.
\end{split}
\end{equation}
Substituting the found relations into (\ref{eq40}) yields
\begin{equation*}\label{eq41}
\begin{split}
4(3\eta^2-\xi^2)\xi^\TOP A_\eta\xi&=\scal{\tau}{\nabla_\eta \theta_3}+
2\scal{A_\eta\xi}{\nabla_\xi \theta_3}+32\xi^\TOP A_\eta^3\xi\\
&-20\eta^2(\xi^\TOP A_\eta\xi)
-4\xi^2(\xi^\TOP A_\eta\xi).
\end{split}
\end{equation*}
Collecting terms in the latter relation by homogeneity, we obtain additionally the following relations:
\begin{equation}\label{er1}
\scal{\tau}{\nabla_\eta\theta_3}=0,
\end{equation}
\begin{equation}\label{er2}
\scal{A_\eta\xi}{\nabla_\xi \theta_3}=0,
\end{equation}
\begin{equation}\label{er3}
\xi^\TOP A_\eta^3\xi=\eta^2(\xi^\TOP A_\eta\xi).
\end{equation}
Since (\ref{er3}) holds for any $\xi$, we get the required Clifford type matrix identity:
\begin{equation}\label{eta}
A_\eta^3=\eta^2 A_\eta.
\end{equation}

Similarly, collecting the terms in (\ref{eq5}) by homogeneity yields
\begin{equation}\label{es1}
|\nabla_\xi \theta_4|^2+|\nabla_\eta \theta_3|^2=16\xi^6,
\end{equation}
\begin{equation}\label{es2}
\scal{\nabla_\xi \theta_4}{\nabla_\xi \theta_3}+ \scal{\nabla_\eta \theta_3}{\nabla_\eta \theta_2}=0,
\end{equation}
\begin{equation}\label{es3}
64(\xi^\TOP A_\eta\xi)^2+2\scal{\nabla_\xi \theta_4}{\nabla_\xi \theta_2}+
|\nabla_\eta \theta_2|^2+|\nabla_\xi \theta_3|^2=48\xi^4\eta^2,
\end{equation}
\begin{equation}\label{es4}
\scal{\nabla_\xi \theta_3}{\nabla_\xi \theta_2}+
\scal{\nabla_\eta \theta_3}{\nabla_\eta \theta_0}=0.
\end{equation}

Now we are ready to characterize quartic eikonal polynomials with lower dimensions $p$ and $q$.

\begin{proposition}\label{pr1}
If $f\in E_{0,n-1}\cup E_{n-1,0} \cup E_{n-2,1}$ then $f$ is a primitive quartic.
\end{proposition}

\begin{proof}
First consider the case $f\in E_{0,n-1}\cup E_{n-1,0}$. Then $qp=0$, thus $\psi\equiv0$. If $q=0$ then $\bar x=\xi$ and by (\ref{theta}), $\theta\equiv\theta_4=\xi^4$. Hence (\ref{reduced}) yields the required property, because
$$
f=x_n^4+2x_n^2\, \bar x^2+\bar x^4=(x_n^2+\bar x^2)^2= x^4\equiv h_{4,V}.
$$
If $p=0$ then $\bar x=\eta$ and (\ref{theta}) yields $\theta =\theta_0=\bar x^4$. This again shows that $f$ is primitive:
$$
f=x_n^4-6x_n^2\, \bar x^2+\bar x^4\equiv h_{4,H},
$$
where $H=\R{}e_1$.

It  remains to consider the case $q=1$. Then $\eta\equiv\eta_1$, so that $\theta_3=g_1(\xi)\eta_1$, where $g_1$ is a cubic form in $\xi$. Furthermore, (\ref{er1}) implies $(\xi^\TOP A_1\xi)g_1(\xi)\equiv 0$. Since there are no zero divisors in the polynomial ring $\R{}[\xi_1,\ldots,\xi_p]$, either $\xi^\TOP A_1\xi$ or $g_1(\xi)$ must be identically zero.
If $\xi^\TOP A_1\xi\equiv 0$ then  $A_1=0$ and by virtue of (\ref{theta})
\begin{equation}\label{last}
\theta_4=\xi^4, \quad \theta_3=g(\xi)\eta_1, \quad \theta_2=-6\xi^2\eta_1^2, \quad \theta_0=\eta_1^4.
\end{equation}
By using (\ref{eq5}) and the last identity one finds $16\xi^6+g_1^2=16\xi^6,$ hence $g_1\equiv 0$. Therefore, by virtue of (\ref{last}) and (\ref{reduced})
\begin{equation*}
\begin{split}
f&=x_n^4+2x_n^2(\xi^2-3\eta_1^2)+\xi^4-6\xi^2\eta_1^2+\eta_1^4\\
&=\eta_1^4-6(x_n^2+\xi^2)\eta_1^2+(x_n^2+\xi^2)^2\equiv h_{4,H},
\end{split}
\end{equation*}
where $H$ is a one-dimensional subspace spanned on the coordinate vector corresponding to $\eta_1$.

Now suppose that $A_1\ne 0$. Then $g_1\equiv 0$ and (\ref{eta}) implies that $A_1^3=A_1$. Hence $A_1$ is a symmetric matrix with eigenvalues $\pm1$ and $0$. Denote by $L=L^+\oplus L^-\oplus L^0$ the corresponding  eigen decomposition of $L$. Since $A_1\ne 0$, the subspace $L^+\oplus L^-$ is nontrivial. We claim that $L^0=\{0\}$. Indeed, let $\xi=u\oplus v\oplus w$ be the decomposition of an arbitrary $\xi\in L$ according to the eigen decomposition of $L$. Then $\xi^\TOP A_1\xi=u^2-v^2$,  and by (\ref{theta}), $\theta_4=(u^2+v^2+w^2)^4-2(u^2-v^2)^2$. Since $\theta_3= g_1\eta_1\equiv 0$, we obtain by virtue of (\ref{es1})
\begin{equation*}
\begin{split}
0\equiv  |\nabla_\xi \theta_4|^2-16\xi^6=-64(u^2-v^2)^2w^2,
\end{split}
\end{equation*}
which implies that $w$ the zero vector. Thus $L^0$ is trivial and it follows that $A_1^2=\mathbf{1}$.
By virtue of  (\ref{theta}), we obtain
\begin{equation*}
\begin{split}
\theta&=\xi^4-2(\xi^\TOP A_1\xi)^2+8(\xi^\TOP A_1^2\xi)\eta_1^2-6\xi^2\eta^2+\eta_1^4\\
&=(u^2+v^2)^2-2(u^2-v^2)^2+8(u^2+v^2)\eta_1^2-6(u^2+v^2)\eta_1^2+\eta_1^4\\
&=(u^2+v^2+\eta_1^2)^2-2(u^2-v^2)^2,
\end{split}
\end{equation*}
hence
\begin{equation*}
\begin{split}
f&=x_n^4+2x_n^2(u^2+v^2-3\eta_1^2)+8x_n\eta_1(u^2-v^2)+(u^2+v^2+\eta_1^2)^2-2(u^2-v^2)^2\\
&=(x_n^2+u^2+v^2+\eta_1^2)^2-2(u^2-v^2-2x_n\eta_1)^2.
\end{split}
\end{equation*}
By making an orthogonal change of coordinates $\eta_1=\frac{1}{\sqrt{2}}(t_1+t_2)$ and $x_n=\frac{1}{\sqrt{2}}(t_2-t_1)$, we obtain
\begin{equation*}
\begin{split}
f&=(u^2+t_1^2+v^2+t_2^2)^2-2(u^2+t_1^2-v^2-t_1^2)^2\\
&=-(u^2+t_1^2)^2+6(u^2+t_1^2)(v^2+t_1^2)-(v^2+t_2^2)^2\equiv -h_{4,H},
\end{split}
\end{equation*}
where $H$ is spanned on $u$ and $t_1$. It follows that $f$ is primitive.
\end{proof}

\section{Harmonicity of $\theta_3$}

In what follows we shall assume that $f\in E_{p,q}$ with $p\ge 1$ and $q\ge 2$. We need the following general fact on the matrix solutions of (\ref{eta}).

\begin{lemma}\label{lem2}
Let $A_i\in \R{p\times p}$, $1\le i\le q$, be symmetric matrices satisfying $A_\eta^3=\eta^2 A_\eta$, where $A_\eta= \sum_{i=1}^q\eta_iA_i$. If $q\ge2$ then
\begin{itemize}
\item[$(\mathrm{i})$]
there are nonnegative integers $\nu$ and $\mu$, $2\nu+\mu=p$, such that for any  $\eta\in \R{q}$, $\eta\ne 0$ and for any $i$, $1\le i\le q$,  the matrices $\frac{1}{|\eta|}A_\eta$ and $A_i$ have eigenvalues $\pm1$ of multiplicity $\nu$ and the eigenvalue $0$ of multiplicity $\mu$;
\item[$(\mathrm{ii})$]
all $A_i$ are trace free;
\item[$(\mathrm{iii})$]
$A_i^3=A_i$ for any $i$.
\item[$(\mathrm{iv})$]
for any two vectors $t,s\in \R{q}$ such that $\scal{s}{t}=0$,
\begin{equation*}\label{iv}
A_s^2A_t+A_sA_tA_s+A_tA_s^2=s^2 A_t.
\end{equation*}
\end{itemize}
\end{lemma}

\begin{proof}
The assumption $A_\eta^3=\eta^2 A_\eta$ implies that for any $\eta\ne 0$, the eigenvalues of $A_\eta$ are $\pm|\eta|$ and $0$. Denote by $\nu^{\pm}(\eta)$ and $\mu(\eta)$ the multiplicities of $\pm|\eta|$ and $0$ respectively. For $q\ge 2$ the unit sphere $S=\{\eta\in M: |\eta|=1\}$ is connected and the traces
$$
\trace A_\eta=\nu^+(\eta) - \nu^-(\eta), \qquad
\trace A^2_\eta=\nu^+(\eta) + \nu^-(\eta),
$$
as functions defined on the unit sphere $S$ are continuous and integer-valued, thus they must be constants. It follows that $\nu^+(\eta)$ and $\nu^-(\eta)$ are also constants. In particular, $\mu(\eta)=q-\nu^+(\eta) - \nu^-(\eta)$  is a constant.

We  notice that
\begin{equation}\label{trace}
\trace A_\eta\equiv \sum_{i=1}^q \eta_i\trace A_i=|\eta|(\nu^+ - \nu^-).
\end{equation}
Since $|\eta|$ is non-linear for $q\ge 2$, we conclude that  $\nu^+=\nu^-$. We denote by $\nu$ the common value of $\nu^\pm$. Then (\ref{trace}) yields $\trace A_\eta=0$, hence $\trace A_i=0$ for all $i$.  Since $A_i=A_{e_i}$ and  $|e_i|=1$, we conclude that $A_i$ has eigenvalues $\pm1$ and $0$ of multiplicities $\nu$ and $\mu=q-2\nu$ respectively. This proves (i)--(iii).
In order to prove (iv), we put $\eta=s+\lambda t$ in (\ref{eta}) and identify the coefficient of $\lambda$.  The proposition is proved.
\end{proof}

\begin{corollary}\label{cor:A}
With any solution $f\in E_{p,q}$, $q\ge2$, one can associate nonnegative integers  $\nu$ and $\mu=p-2\nu$ such that the matrix $A_\eta$ defined by (\ref{psi}) is  similar to the diagonal trace free matrix
\begin{equation}\label{similar}
A_\eta\sim |\eta|\left(
        \begin{array}{ccc}
          \mathbf{1}_\nu &  &  \\
           & -\mathbf{1}_\nu & \\
           &  & \mathbf{0}_\mu\\
        \end{array}
      \right),
\end{equation}
where $\mathbf{1}_\nu$ and $\mathbf{0}_\mu$ stands for the unit matrix and the zero matrix of sizes $\nu\times \nu$ and $\mu\times \mu$ respectively and the elements not shown are all zero.
\end{corollary}

\begin{remark}\label{rem1}
For the sake of convenience, we shall indicate the situation in Corollary~\ref{cor:A}  by writing $f\in E_{p,q}^{\nu}$.
\end{remark}

\begin{lemma}\label{lem3}
If $f\in E_{p,q}^{\nu}$, $q\ge 2$, $p\ge1$, then
\begin{equation}\label{lapl3}
\begin{split}
\Delta_x f=4(p-3q+3)(x_n^2+\xi^2)+4(8\nu-1+q-3p)\eta^2+\Delta_\xi\theta_3.
\end{split}
\end{equation}
\end{lemma}

\begin{proof}
By virtue of (\ref{reduced}), $f=x_n^4+2(\xi^2-3\eta^2)x_n^2+8(\xi^\TOP A_\eta\xi)x_n+\theta$.
Applying (ii) in Lemma~\ref{lem2} we find for the Laplacian
\begin{equation}\label{lapl}
\begin{split}
\Delta_x f&\equiv (\partial_{x_n}^2 +\Delta_\xi+\Delta_\eta)f\\
&=12x_n^2+4(\xi^2-3\eta^2)+4x_n^2(p-3q)+16x_n\trace A_\eta+\Delta_\xi\theta+\Delta_\eta\theta\\
&=4(p-3q+3)x_n^2+4(\xi^2-\eta^2)+\Delta_\xi\theta+\Delta_\eta\theta.
\end{split}
\end{equation}
By (\ref{similar}), $\trace A_\eta^2=2\nu\eta^2$, hence we find from (\ref{theta})
\begin{equation*}\label{lapl1}
\begin{split}
\Delta_\xi\theta
&=4(p+2)\xi^2+(32\nu-12p)\eta^2-16\sum_{i=1}^q \xi^\TOP A_i^2\xi.
\end{split}
\end{equation*}
Similarly we find $\Delta_\eta\theta=\Delta_{\eta}\theta_3+16\sum_{i=1}^q \xi^\TOP A_i^2\xi- 12q\xi^2+4(q+2)\eta^2$,
and combining these formulas with (\ref{lapl}) yields (\ref{lapl3}).
\end{proof}

Now we are ready to proof the main result of this section.

\begin{proposition}\label{th:eta}
If $f\in E_{p,q}^{\nu}$, $q\ge 2$, $p\ge1$, then $\Delta_\eta\theta_3=0$.
\end{proposition}

\begin{proof}
Write $\theta_3=8\sum_{i=1}^q g_i(\xi)\eta_i$, where $g_i$ are cubic forms in $\xi$. It suffices to show that $\Delta_\xi g_1\equiv 0$.
To this end, let us consider the eigen decomposition $L=L_1^+\oplus L_1^-\oplus L_1^0$ associated with $A_1$ according Lemma~\ref{lem2}, and decompose the cubic form $g_1$ into the corresponding homogeneous parts:
\begin{equation}\label{eqg}
g_1=\sum_{s}G_{s}, \qquad G_s\in {u}^{s_1}\otimes {v}^{s_2} \otimes {w}^{s_3},
\end{equation}
where $s=(s_1,s_2,s_3)$, $s_1+s_2+s_3=3$, and $\xi =u\oplus v\oplus w$, $u\in L^+$, $v\in L^-$, $w\in L^0$.

By (\ref{er2}), $\scal{A_1\xi}{\nabla_\xi g_1}=0$. We have $A_1\xi=u-v$, so that applying  (\ref{eqg}) and the homogeneity of each $G_s$, we obtain
$$
0=\sum_{s}\scal{\nabla_\xi G_s}{A_1\xi}=\sum_{s}(\scal{\nabla_u G_s}{u}-\scal{\nabla_v G_s}{v})
=\sum_{s}(s_1-s_2)G_s.
$$
Since non-zero components $G_s$ are linear independent, we have $s_1=s_2$. This yields
\begin{equation}\label{eks}
s\in \{(1,1,1), \,(0,0,3)\}
\end{equation}

On the other hand, we infer from (\ref{es4}) by virtue of (\ref{grad}) and (\ref{theta})  that
$$
\scal{\nabla_\xi \theta_3}{A_\eta^2 \xi}=2\eta^2 \theta_3.
$$
By identifying the coefficients of $\eta_1^3$, $\scal{\nabla_\xi g_1}{A_1^2 \xi}=2g_1$, thus, by virtue of
$A_1^2\xi=u+v$,
$$
2g_1=\scal{\nabla_\xi g_1}{A_1^2 \xi}\equiv \sum_{s}\scal{A_1^2\xi}{\nabla_\xi G_s}
=\sum_{s}(s_1+s_2)G_s.
$$
Comparing with (\ref{eqg}) implies $s_1+s_2=2$, hence by (\ref{eks}) we obtain $s=(1,1,1)$, i.e.
$g_1\in u\otimes v\otimes w$ is a trilinear form. It follows  that $\Delta_\xi g_1=0$. The theorem is proved.
\end{proof}

\begin{corollary}\label{cor:last}
If $f\in E_{p,q}^{\nu}$, $q\ge 2$, $p\ge1$, then
\begin{equation}\label{lapl4}
\begin{split}
\Delta_x f=4(p-3q+3)(x_n^2+\xi^2)+4(8\nu-1+q-3p)\eta^2.
\end{split}
\end{equation}
\end{corollary}

\section{Proof of Theorem~\ref{th:polar}}
Suppose that $f$ is a non-primitive eikonal quartic. Then by Proposition~\ref{pr1}, $f\in E_{p,q}^{\nu}$ with $p\ge 1$ and $q\ge 2$.
Applying  (\ref{grad}) to (\ref{es3}), one obtains the following identity:
\begin{equation}\label{eta2}
\begin{split}
\frac{1}{64}|\nabla_\xi \theta_3|^2&=4\xi^2(\xi^\TOP A_\eta^2 \xi)+4\xi^\TOP A_\tau A_\eta^2 \xi
-3\eta^2(\xi^\TOP A_\tau \xi)-(\xi^\TOP A_\eta \xi)^2\\
&-4\sum_{i=1}^q(\xi^\TOP A_iA_\eta\xi)^2.
\end{split}
\end{equation}

 By Proposition~\ref{th:eta}, $\Delta_\xi\theta_3^2=2|\nabla_\xi\theta_3|^2,$ hence  (\ref{eta2}) implies
\begin{equation}\label{delta1}
\begin{split}
\frac{1}{128}\Delta_\eta \Delta_\xi\theta_3^2&=
8\xi^2 (\xi^\TOP B \xi) +8\xi^\TOP A_\tau B \xi
-6q(\xi^\TOP A_\tau \xi)\\
&-\Delta_\eta(\xi^\TOP A_\eta \xi)^2-4\Delta_\eta\sum_{i=1}^q(\xi^\TOP A_iA_\eta\xi)^2,
\end{split}
\end{equation}
where $B:=\sum_{i=1}^q A_i^2$. We have
$$
\Delta_\eta(\xi^\TOP A_\eta \xi)^2=2\sum_{i=1}^q (\xi^\TOP A_i \xi)^2=2\sum_{i=1}^q\tau_i^2\equiv 2\tau^2,
$$
and similarly
$$
\Delta_\eta\sum_{i=1}^q(\xi^\TOP A_iA_\eta\xi)^2
=2\sum_{i,j=1}^q (\xi^\TOP A_iA_j \xi)^2.
$$
Taking into account that $\xi^\TOP A_\tau \xi=\tau^2$, we get from (\ref{delta1})
\begin{equation}\label{summary}
\frac{1}{128}\Delta_\eta \Delta_\xi\theta_3^2=
8\xi^2 \cdot \xi^\TOP B \xi +8\xi^\TOP A_\tau B \xi-8\sum_{i,j=1}^q (\xi^\TOP A_iA_j \xi)^2
-(6q+2)\tau^2.
\end{equation}

On the other hand, $\Delta_\eta\theta_3=0$ because $\theta_3$ is linear in $\eta$, hence
\begin{equation}\label{summary1}
\frac{1}{128}\Delta_\eta \Delta_\xi\theta_3^2=
\frac{1}{128}\Delta_\xi \Delta_\eta\theta_3^2=\frac{1}{64}\Delta_\xi |\nabla_\eta\theta_3|^2.
\end{equation}
Applying (\ref{es1}) and (\ref{grad}), we find
\begin{equation}\label{th3eta}
|\nabla_\eta \theta_3|^2=16\xi^6-|\nabla_\xi \theta_4|^2=64(\xi^2 \cdot\xi^\TOP A_\tau\xi-\xi^\TOP A_\tau^2\xi)=
64(\xi^2 \tau^2-\xi^\TOP A_\tau^2\xi),
\end{equation}
where $\tau_i=\xi^\TOP A_i\xi$.

Our next step is the $\xi$-Laplacian of the right hand side of (\ref{th3eta}). We have $\nabla_\xi\tau_i=2A_i\xi$ and by Lemma~\ref{cor:A},  $\Delta_\xi\tau_i=2\trace A_i=0$. Therefore
$$
\Delta_\xi\tau^2=2\sum_{i=1}^q|\nabla_\xi\tau_i|^2=8\sum_{i=1}^q\xi^\TOP A_i^2\xi\equiv 8\xi^\TOP B\xi,
$$
and $\Delta_\xi (\xi^2 \tau^2)
=
(2p +16)\tau^2+8\xi^2(\xi^\TOP B\xi)
$.
Furthermore,
$$
\xi^\TOP A_\tau^2\xi=\sum_{i,j=1}^q \rho_{ij}\tau_i\tau_j, \qquad \rho_{ij}=\xi^\TOP A_iA_j\xi.
$$
We find $\nabla_\xi\rho_{ij}=(A_iA_j+A_jA_i)\xi$ and $\Delta_\xi\rho_{ij}=2\trace A_iA_j$, hence
\begin{equation*}
\begin{split}
\Delta_\xi(\xi^\TOP A_\tau^2\xi)&=
\sum_{i,j=1}^q\tau_i\tau_j\Delta_\xi\rho_{ij}+4\tau_i\scal{\nabla_\xi\rho_{ij}}{\nabla_\xi\tau_j}+
2\rho_{ij}\scal{\nabla_\xi\tau_{i}}{\nabla_\xi\tau_j}\\
&=2\trace A_\tau^2+8\sum_{i,j=1}^q \xi^\TOP A_j(A_i A_j+A_jA_i)\xi\,\tau_i+8\sum_{i,j=1}^q (\xi^\TOP A_iA_j \xi)^2.
\end{split}
\end{equation*}
By (iii) and (iv) in Lemma~\ref{lem2}, $A_j(A_i A_j+A_jA_i)=(1+2\delta_{ij})A_i-A_i A_j^2,$
where $\delta_{ij}$ is the Kronecker delta, hence
$$
\sum_{i,j=1}^q \xi^\TOP A_j(A_i A_j+A_jA_i)\xi\,\tau_i=\xi^{\TOP}(qA_\tau+2A_\tau-A_\tau B)\xi=(q+2)\tau^2-\xi^{\TOP}A_\tau B\xi.
$$
Since $\trace A_\tau^2=2\nu\tau^2$, we obtain
\begin{equation*}
\begin{split}
\Delta_\xi(\xi^\TOP A_\tau^2\xi)=(4\nu+8q+8)\tau^2-8\xi^{\TOP}A_\tau B\xi+8\sum_{i,j=1}^q (\xi^\TOP A_iA_j \xi)^2.
\end{split}
\end{equation*}
On substituting the found relations into (\ref{th3eta}), we find
\begin{equation*}
\frac{1}{64}\Delta_\xi|\nabla_\eta \theta_3|^2=8\xi^2(\xi^\TOP B\xi)+8\xi^{\TOP}A_\tau B\xi-8\sum_{i,j=1}^q (\xi^\TOP A_iA_j \xi)^2+2(p -4q-2\nu)\tau^2,
\end{equation*}
and on combining the last relation with (\ref{summary}) and (\ref{summary1}), we arrive at
\begin{equation}\label{Main}
(p +1-q-2\nu)\tau^2=0.
\end{equation}

Suppose that $\tau\equiv0$. Then $A_i=0$ for all $i=1,\ldots, q$. It follows from (\ref{eta2}) that $\theta_3\equiv 0$ and on applying (\ref{theta}),
$
\theta\equiv \theta_4+\theta_2+\theta_0=\xi^4-6\xi^2\eta^2+\eta^4.
$
This yields
$$
f=x_n^4+2x_n^2(\xi^2-3\eta^2)+\xi^4-6\xi^2\eta^2+\eta^4
=(x_n^2+\xi^2)^2-6\eta^2(x_n^2+\xi^2)+\eta^4\equiv h_{4,H},
$$
where $H$ is spanned on $x_n$ and $\xi$, hence $f$ is primitive eikonal quartic, a contradiction. Thus,  $\tau\not\equiv 0$ and (\ref{Main}) implies $2\nu=p+1-q$. On substituting this into (\ref{lapl4}), one finds
$$
\Delta_x f=4(p-3q+3)(x_n^2+\xi^2+\eta^2)=4(p-3q+3)x^2=8(\nu-q+1)x^2.
$$
In view of $n=2\nu+2q$ one finds by comparing with (\ref{Muntzer1}) that $m_1=q-1$ and $m_2=\nu$ are integers. It follows that $f$ is an isoparametric polynomial. The theorem is proved completely.

\section{Concluding remarks}
After this paper was finished, the author was informed by Professor Tang Zizhou that about the paper  \cite{wang}. In this paper, Q.M.~Wang proves that any smooth solution $f$ of the equation $|\overline{\nabla} f|^2=b(f)$  on $M=\R{n}$ and $M=\mathbb{S}^n$ has isoparametric fibration, i.e. must have only the level sets which are isoparametric on $M$ ($\overline{\nabla}$ denotes the covariant derivative on $M$), see also a recent preprint of R.~Miyaoka \cite{Miyaoka}. The Wang theorem clarify, to some content, the appearance and algebraic structure of the primitive eiconal polynomials $h_{g,H}$  in (\ref{virtue}) above. Indeed, any solution of the eiconal equation $|\nabla_x f|^2=k^2x^{2k-2}$ in $\R{n}$ induces (by the restriction) a transnormal function $F:=f|_{\mathbb{S}^{n-1}}$ on the unit sphere $\mathbb{S}^{n-1}\subset\R{n}$ satisfying the transnormal condition $|\overline{\nabla}F|^2=k^2(1-F^2)$. Thus, according to the Wang theorem, the level sets  $F=c$ are isoparametric submanifolds in $\mathbb{S}^{n-1}$, which agrees with the conclusion of our Theorem~\ref{th:polar} because the level set $h_{g,H}=\cos t$, $t\in \R{}$ splits into standard products of two spheres $\mathbb{S}^{p}(\cos \frac{t+2\pi m}{g})\times \mathbb{S}^{q}(\sin \frac{t+2\pi m}{g})$,  $p=\dim H-1$, $q=n-\dim H-1$, and $m=0,1,\ldots,g-1$. The latter tori are well known to be isoparametric hypersurfaces in  $\mathbb{S}^{n-1}$ with 2 distinct principal curvatures, see for instance \cite{Cecil}.

We  would like to thank Professor Tang Zizhou for many helpful suggestions, and, especially, for bringing to our attention the Wang paper \cite{wang} and also the papers \cite{GeTang}, \cite{Miyaoka}, \cite{Tang3}, \cite{Tang1}, \cite{Tang2} where related questions are treated.


%
\end{document}